\documentclass{amsart}
\usepackage{amsmath}
\usepackage{amssymb}
\usepackage{amsthm}
\usepackage{a4wide}
\usepackage{extarrows}
\usepackage{url}

\usepackage[final]{showlabels}
\usepackage{xcolor}

\newtheorem{thm}{Theorem}[section]
\newtheorem{lem}[thm]{Lemma}
\newtheorem{prop}[thm]{Proposition}
\newtheorem{cor}[thm]{Corollary}

\theoremstyle{definition}
\newtheorem{defn}[thm]{Definition}
\newtheorem{rem}[thm]{Remark}
\newtheorem{quest}[thm]{Question}

\newcommand{\R}{\ensuremath{\mathbb{R}}}
\newcommand{\dlat}{\mathrm{d}}

\DeclareMathOperator{\cl}{cl}
\DeclareMathOperator{\K}{\mathcal{K}}
\DeclareMathOperator{\inr}{r}
\DeclareMathOperator{\vol}{vol}
\def\W{\mathrm{W}}
\DeclareMathOperator{\U}{\mathcal{U}}
\def\Sa{\mathrm{S}}
\def\s{\mathbb{S}}
\def\I{\mathrm{I}}
\def\r{\mathcal{R}}

\def\abs#1{\left| #1\right|}

\numberwithin{equation}{section}

\begin{document}

\title{On the monotonicity of the isoperimetric quotient for parallel bodies}
\author{Christian Richter}
\address{Institute of Mathematics, Friedrich Schiller University, 07737 Jena, Germany}
\email{christian.richter@uni-jena.de}

\author{Eugenia Saor\'\i n G\'omez}
\address{ALTA institute for Algebra, Geometry, Topology and their
Applications, University of Bremen, 28334 Bremen, Germany}
\email{esaoring@uni-bremen.de}

\thanks{The second author is partially supported by MICINN/FEDER
project PGC2018-097046-B-I00 and Fundaci\'on S\'eneca project
19901/GERM/15}

\date{\today}

\begin{abstract}
The isoperimetric quotient of the whole family of inner and outer parallel bodies of a convex body is shown to be decreasing in the parameter of definition of parallel bodies, along with a characterization of those convex bodies for which that quotient happens to be constant on some interval within its domain. 
This is obtained relative to arbitrary gauge bodies, having the classical Euclidean setting as a particular case.
Similar results are established for different families of Wulff shapes that are closely related to parallel bodies. These give rise to solutions of isoperimetric-type problems. Furthermore, new results on the monotonicity of quotients of other quermassintegrals different from surface area and volume, for the family of parallel bodies, are obtained.
\end{abstract}



\maketitle

\section{Introduction}

Inner parallel bodies of convex bodies have been object of recent studies with different flavors \cite{Domokos, GK, HCS, HCSq, HW, Larson, RV}. Classical references on inner parallel bodies (e.g.\ \cite{Bol43, Di, Di49, Ha55, Ha57, SYphd}) along with their role in the proofs of fundamental results in the theory of convex bodies make inner parallel bodies an essential object within classical convex geometry \cite[Section 7.5]{Sch}.
Inner parallel bodies and their properties were thoroughly studied by Bol \cite{Bol43}, Dinghas \cite{Di} (see also \cite{Ha55, Ha57}) and later by Sangwine-Yager \cite{SYphd}.

Let $\K^n$ denote the family of convex bodies in $\R^n$, i.e., of nonempty compact convex
subsets of the Euclidean space $\R^n$, and let $\K^n_n$ be its subfamily of convex bodies with nonempty interior. Let $B_n$ be the $n$-dimensional unit ball and $\s^{n-1}$ the corresponding unit sphere. The volume of a convex body $K\subseteq\R^n$,
i.e., its $n$-dimensional Lebesgue measure, is denoted by $\vol(K)$. The measure of its boundary, i.e., its surface area or $(n-1)$-dimensional Hausdorff measure, is denoted by $\Sa(K)$.

For $K\in\K_n^n$, the classical \textit{isoperimetric quotient} is the ratio 
\begin{equation}\label{e:class isop quot}
\I(K)=\frac{\Sa(K)^{n}}{\vol(K)^{n-1}}.
\end{equation}

Let $K,E\in\K^n$. The {\em inradius $\inr(K;E)$ of $K$ relative to $E$} is the largest possible factor of a homothety mapping $E$ into $K$, i.e., 
\[
\inr(K;E)=\sup\{r\geq 0:\text{ there is } x\in\R^n\text{ with } x+r\,E\subseteq K\}.
\]

For $-\inr(K;E)\leq\lambda\leq 0$, the {\em inner parallel body $K_\lambda$ of $K$ relative to $E$ at distance
$\abs{\lambda}$} is the Minkowski difference of $K$ and $\abs{\lambda}E$, i.e.,
\[
K_{\lambda}=K\sim\abs{\lambda}E=\{x\in\R^n:\abs{\lambda} E+x\subseteq K\}\in\K^n.
\] 
If $E=B_n$, then $K_{-\inr(K;B_n)}$ is the set of
incenters of $K$. For any $E\in\K^n$, the set $K_{-\inr(K;E)}$ has dimension strictly less than $n$ (see \cite[p.~59]{BF}). The inner parallel sets complement the family of outer parallel sets $K_\lambda$, $\lambda \ge 0$, relative to $E$ at distance $\lambda$, that are defined as the Minkowski sums $K_\lambda=K+\lambda E$.
 
We prove that $\I(K)\geq \I(K_\lambda)$ for $- \inr(K;E) < \lambda < 0$ by investigating the behaviour of the isoperimetric quotient as a function of the real parameter $\lambda$. Our main result in this direction is the following (see Corollary~\ref{c: class isop quot}), extending the well-known result for outer parallel bodies \cite[Remark 4.4]{HCSq} to inner ones.

\begin{thm}
Let $K\in\K_n^n$. Then the isoperimetric quotient $\I(K_\lambda)$ is monotonically decreasing for $\lambda \in (-\inr(K;B_n),\infty)$.
\end{thm}

We will also analyze the situation where the function $\lambda \mapsto\I\left(K_{\lambda}\right)$ is constant in some subinterval of $(-\inr(K;B_n), \infty)$, characterizing the convex bodies $K\in\K^n_n$ for which this behaviour occurs.

We notice that the surface area is implicitly defined with respect, or relative, to the Euclidean unit ball, see Section \ref{s:background} for details. Indeed, the two magnitudes, volume and surface area, involved in the isoperimetric quotient are particular cases of the so-called quermassintegrals of a convex body, and can be defined relative to a fixed body $E \in \mathcal{K}_n^n$, for which we refer the reader to Section~\ref{s:background} and the references therein.

In this note, for a fixed reference body $E\in\K^n_n$ - usually called {\em gauge}, we investigate the behaviour of the relative isoperimetric quotient function $\frac{\Sa\left(K_{\lambda};E\right)^n}{\vol\left(K_{\lambda}\right)^{n-1}}$ (see Section \ref{s:background} for details) of the family of parallel bodies of the convex body $K\in\K_n^n$ with respect to $E$. 
In \cite{HCSq} (cf.\ \cite{Domokos}) the authors prove that, under certain boundary restrictions of the involved convex bodies or under certain differentiability conditions of the involved magnitudes, respectively, the relative isoperimetric quotient $\frac{\Sa\left(K_{\lambda};E\right)^n}{\vol\left(K_{\lambda}\right)^{n-1}}$ is monotonically decreasing for $\lambda \in (-\inr(K;E),\infty)$.  For $\lambda\geq 0$, i.e., for outer parallel bodies, this is a direct consequence of the relative Steiner formula \eqref{e:steiner_rel} \cite[Remark~4.4]{HCSq}. We point out that in the planar case $K,E \in \mathcal{K}_2^2$ the monotonicity in the whole domain $(-\inr(K;E),\infty)$ seems to be treated as folklore, and to the best of our knowledge, there is no precise reference containing this result.

Here we prove that no assumption on the convex bodies $K,E\in\K^n_n$ is necessary to have monotonicity on the total domain $(-\inr(K;E),\infty)$. Moreover, we characterize all convex bodies for which the quotient is constant in some interval of its domain (see Theorem~\ref{thm:overline{I}}).

In the above result, the family $(K_\lambda)_{-\inr(K;E) \le \lambda \le 0}$ of inner parallel bodies is {\em extended} by the family $(K_\lambda)_{\lambda \ge 0}$ of outer parallel bodies. In Section~\ref{sec:  Wulff} we introduce other natural extensions of $(K_\lambda)_{-\inr(K;E) \le \lambda \le 0}$ to parameters $\lambda \ge 0$, based on particular Wulff shapes. Also the isoperimetric quotients of these families are monotonically decreasing (see Theorem~\ref{thm:  Wulff}). This allows us to characterize isoperimetrically optimal bodies $K$ with prescribed outer normals, again relative to a gauge body $E$ (see Section~\ref{sec:  Wulff}).

Finally, we investigate the behaviour of further quotients of magnitudes for parallel bodies intimately connected to the isoperimetric quotient, both for the classical and relative cases
(see Section~\ref{sec:  other qmi}). This continues work of \cite{HCSq}.


\section{Background}\label{s:background}

We write $K_{\lambda}$ to denote the inner and outer parallel bodies
of $K\in\K^n$ relative to the gauge body $E\in\K^n_n$,
\[
K_{\lambda}=\left\{\begin{array}{ll}
K\sim\abs{\lambda}E & \quad\text{ for }-\inr(K;E)\leq\lambda\leq 0,\\[1mm]
K+\lambda E & \quad\text{ for }\;0\leq\lambda<\infty.
\end{array}\right.
\]

For $K\in\K^n$ and
$u\in\R^n$, $h_K(u)=\sup\bigl\{\langle x,u\rangle:x\in
K\bigr\}$ denotes the {\em support function} of $K\in\K^n$, where $\langle \cdot,\cdot \rangle$ stands for the standard Euclidean scalar product in $\R^n$ (see e.g.\ \cite[Section 1.7]{Sch}). For $u \in\R^n\setminus\{0\}$ and $\alpha \in \R$, we define the halfspace $H^-_{u,\alpha}=\{x \in \R: \langle x,u\rangle \le \alpha\}$ \cite[p.\ xx]{Sch}. Then the above parallel bodies $K_\lambda$ can be equivalently defined as 
\begin{equation}\label{eq:  Kl intersection}
K_{\lambda}=\bigcap_{u \in \s^{n-1}} H^-_{u,h_K(u)+\lambda h_E(u)}
\end{equation}
for all $\lambda \in [-\inr(K;E),\infty)$ \cite[Lemma~4.1]{SYphd}, \cite[formula (3.19) and Theorem 1.7.5(a)]{Sch}.

The so-called {\em relative Steiner formula} states that the volume of the
outer parallel body $K+\lambda E$ is a polynomial of degree $n$ in
$\lambda\geq 0$,
\begin{equation}\label{e:steiner_rel}
\vol(K+\lambda E)=\sum_{i=0}^n \binom{n}{i}\W_i(K;E)\lambda^i.
\end{equation}
The coefficients $\W_i(K;E)$ are called the {\em relative
quermassintegrals} of $K$, and they are just a special case of the more
general {\em mixed volumes}, for which we refer to \cite[Section~5.1]{Sch} and \cite[Sections 6.2, 6.3]{Grub} (where \cite{Sch} uses the notation $V_{(i)}(K,E)$ for $W_i(K;E)$). In particular, we have $\W_0(K;E)= \vol(K)$ and
$\W_n(K;E)=\vol(E)$. 

Applying \eqref{e:steiner_rel} to both sides of $K+(\lambda+\mu)E=(K+\lambda E)+\mu E$ for $\lambda,\mu \ge 0$ and equating the coefficients of $\mu^i$ in both polynomials, we obtain the values of the relative $i$-th quermassintegrals of $K+\lambda E$, namely
\begin{equation}\label{e:steinerW_i}
\W_i(K+\lambda E;E)=\sum_{k=0}^{n-i} \binom{n-i}{k}\W_{i+k}(K;E)\lambda^k,
\end{equation}
for $\lambda\geq 0$ and $i=0,\dots,n$ (cf.\ \cite[Theorem 6.14]{Grub}).

If $E=B_n$, the polynomial on the right-hand side of
\eqref{e:steiner_rel} becomes the classical {\em Steiner polynomial}, see \cite{Ste40}. Then the coefficient $\binom{n}{1}\W_1(K;B_n)$ in \eqref{e:steiner_rel} happens to be the surface area $\Sa(K)$ if $K \in \mathcal{K}_n^n$. This motivates the definition of the {\em surface area of $K$ relative to $E \in \mathcal{K}_n^n$} by 
\begin{equation}\label{e: rel Sa}
\Sa(K;E)=n\W_1(K;E)
\end{equation}
(see e.g.\ \cite[Section 5.1.2]{Ha57}) and the introduction of the relative version of \eqref{e:class isop quot}.

\begin{defn}[\cite{HCSq}]\label{isop quotient}Let $K,E\in\K_n^n$. The {\em isoperimetric quotient of $K$ relative to $E$} is defined as
\begin{equation*}
\label{e: isop quotient}
\I(K;E)=\frac{\Sa(K;E)^{n}}{\vol(K)^{n-1}}.
\end{equation*}
The function
\begin{equation}\label{e: isop quotient function}
\I(\lambda)=\I(K_\lambda;E),
\end{equation}
$\lambda \in (-\inr(K;E),\infty)$, is called the {\em isoperimetric quotient function of $K$ with respect to $E$}.
\end{defn}

Note that $\I(K;E)$ is invariant under homotheties of $K$, since the $i$-th relative quermassintegral is positively homogeneous of degree $n-i$, i.e., $\W_i(\mu K;E)=\mu^{n-i}\W_i(K;E)$ for any $\mu\geq0$.

Our goal is to analyze the behaviour of \eqref{e: isop quotient function}, focusing on the question whether it is a monotonic function in (some parts of) its domain. 

A first naive approach to this matter is the natural question whether there are convex bodies $K,E\in\K^n_n$ for which the isoperimetric quotient function is constant in $(-\inr(K;E),0]$. We observe that $K_{-\inr(K;E)}$ has no interior points, thus we need to exclude it.

Tangential bodies provide us with a positive answer to this question.

\begin{defn}[{\cite[p.\ 149, text preceding Lemma 3.1.14]{Sch}\label{d: tang body}}]
Let $K,E \in \K_n^n$ be such that $E \subseteq K$. Then the body $K$ is a {\em tangential body of $E$} if and only if through each boundary point of $K$ there exists a supporting hyperplane of $K$ that also supports $E$.
\end{defn}

Indeed, inner parallel bodies and tangential bodies happen to be intrinsically connected by means of a homothety relation. The following result enlightens the close connection between inner parallel bodies and tangential bodies, providing us with a constant isoperimetric quotient function on the range $(-\inr(K;E),0)$ of inner parallel bodies.

\begin{thm}[{\cite[Lemma 3.1.14]{Sch}}]
\label{th: Sch tang}
Let $K,E\in\K^n_n$ be convex bodies, and let $\lambda\in\bigl(-\inr(K;E),0\bigr)$. Then
$K_{\lambda}$ is homothetic to $K$ if and only if $K$ is homothetic to a
tangential body of $E$.
\end{thm}

We notice, that if $K$ is a tangential body of $E$, then $\inr(K;E)=1$.
It follows that the isoperimetric quotient function $\I(\lambda)$ is  constant for $-\inr(K;E)=-1<\lambda\leq 0$ if $K$ is a tangential body of $E$.

Although the relative quermassintegral $\W_i(\cdot;E):\K^n\rightarrow \R$, $i=0,\dots,n-1$, is, in general, not linear under Minkowski addition and scalar multiplication, the $(n-i)$-th root of $\W_i(\cdot;E)$ is concave. In particular, the $n$-th root of the volume is a concave function. 
This is a consequence of the fundamental Aleksandrov--Fenchel inequality \cite[Theorem 7.3.1]{Sch}.

\begin{thm}[{\cite[Chapter~7]{Sch}}] (General Brunn--Minkowski theorem for quer\-mass\-integrals).\label{gral BM th}
Let $K,L,E\in\K_n^n$, $0\leq i\leq n-1$, and let $\mu\in[0,1]$. Then,
\begin{equation}\label{BM ineq quermass}
\W_i\bigl(\mu K+(1-\mu)L;E\bigr)^\frac{1}{n-i}\geq \mu\W_i(K;E)^\frac{1}{n-i}+(1-\mu)\W_i(L;E)^\frac{1}{n-i}.
\end{equation}

For $i=0$, i.e., for the volume, equality holds if and only if $K$ and $L$ are homothetic.

If $E$ is smooth and $0 \le i \le n-2$, equality holds if and only if $K$ and $L$ are homothetic.
\end{thm}

Here, a convex body is said to be smooth if it has only one supporting hyperplane at every point of its boundary.
The inequality in Theorem~\ref{gral BM th} can be found in \cite[Theorem 7.4.5]{Sch}, the characterization of equality for $i=0$ is given in \cite[Theorem~7.1.1]{Sch}, and the equality case for smooth $E$ is a consequence of \cite[Theorems 7.4.6 and 7.6.9]{Sch}.

To prove the monotonicity of the quotient function $\I(\lambda)$, the derivative of the volume $\vol(K_\lambda)$, as a function of $\lambda \in (-\inr(K;E),\infty)$, turns out to be crucial.
In order to deal with derivatives of the quermassintegrals $\W_i(K_\lambda;E)$, in particular of the volume $\vol(K_\lambda)=W_0(K_\lambda;E)$, we need the following results.

\begin{lem}[{\cite[Lemma 3.1.13]{Sch}}]\label{fact:parallel_concave}
For all $K,E \in \mathcal{K}_n^n$,
$\lambda_0,\lambda_1\in\bigl[-\inr(K;E),\infty\bigr)$ and $\mu\in[0,1]$,
\begin{equation*}
\mu K_{\lambda_0}+ (1-\mu) K_{\lambda_1} \subseteq K_{\mu\lambda_0+(1-\mu)\lambda_1},
\end{equation*}
i.e., the family $(K_\lambda)_{\lambda \ge -\inr(K;E)}$ of all parallel bodies is concave.
\end{lem}

Given fixed $K,E \in \mathcal{K}_n^n$, we are interested in differentiability properties of the functions
\[
\W_i(\lambda)=\W_i(K_\lambda;E),
\]
$\lambda \in (-\inr(K;E),\infty)$, for $i=0,\ldots,n-1$. In the sequel we write $\frac{\dlat^-}{\dlat\lambda}f(\lambda)$, $\frac{\dlat^+}{\dlat\lambda}f(\lambda)$ and $\frac{\dlat}{\dlat\lambda}f(\lambda)$ for the left, right and both-sided derivative of a function $f$, respectively, implicitly stating that that quantity exists.

Theorem~\ref{gral BM th} and Lemma~\ref{fact:parallel_concave} yield
\begin{equation}\label{e:concav Wi}
\frac{\dlat^-}{\dlat\lambda}\W_i(\lambda)\geq\frac{\dlat^+}{\dlat\lambda}\W_i(\lambda)\geq(n-i)\W_{i+1}(\lambda)
\end{equation}
for all $K,E \in \mathcal{K}_n^n$, $i=0,\ldots,n-1$ and $\lambda \in (-\inr(K;E),\infty)$
(see \cite[formula (1.5)]{HCS}, \cite[p.\ 439, Note 7]{Sch}).

For $i=0$, i.e., for the case of the volume, even more is known. 

\begin{prop}
\label{p: diff vol}
Let $K\in \K^n$, $E\in\K^n_n$.
Then, for $-\inr(K;E)\le\lambda<\infty$, the function $\vol(\lambda)=\vol(K_\lambda)$ is differentiable and satisfies
\begin{equation}\label{e: der vol1}
\frac{\dlat}{\dlat\lambda}\vol(\lambda)=n\W_1(\lambda)
\end{equation}
with only right derivative for $\lambda=-\inr(K;E)$.
\end{prop}

For $\lambda \ge 0$, equation \eqref{e: der vol1} (with right derivative at $\lambda=0$) is a consequence of \eqref{e:steiner_rel}. For $-\inr(K;E) \le \lambda \le 0$ (with left derivative at $\lambda=0$), we refer to \cite{Bol43,Mat}.

For the particular case $E=B_n$, Hadwiger \cite[p.\ 207, formula (30)]{Ha57} has shown that $\frac{\dlat}{\dlat\lambda}\vol(\lambda)=\Sa(K_\lambda)$
for all $\lambda \in (-\inr(K;B_n),\infty)$. Since $n\W_1(K;E)=\Sa(K;E)$, \eqref{e: der vol1} amounts to
\begin{equation}\label{e: der vol}
\frac{\dlat}{\dlat\lambda}\vol(\lambda)=\Sa(K_\lambda;E).
\end{equation}

\begin{rem}
For $\lambda\ge 0$, the differentiability of all quermassintegrals at $\lambda$ (with right derivative for $\lambda=0$) follows from \eqref{e:steinerW_i}, along with the equality $\frac{\dlat}{\dlat\lambda}\W_i(\lambda)=(n-i)\W_{i+1}(\lambda)$. 
\end{rem}

The question for which convex bodies there is equality in \eqref{e:concav Wi} on the domain $-\inr(K;E)<\lambda<\infty$, in particular at $\lambda=0$, is not completely understood yet. We refer the reader to \cite{HCS} for some results in this direction.


\section{Monotonicity of the isoperimetric quotient for parallel bodies}

In this section we will prove that the isoperimetric quotient function is decreasing. Our result is prepared by some auxiliary statements. The first one is an immediate consequence of formula \eqref{e: der vol} or Proposition~\ref{p:  diff vol}.

\begin{lem}\label{l:der vol^1/n}
Let $K,E \in \mathcal{K}_n^n$ and $\lambda \in (-\inr(K;E),\infty)$. Then
\begin{equation}\label{e:der vol^1/n}
\frac{\dlat}{\dlat \lambda} \left(\vol(\lambda)^{\frac{1}{n}}\right)=\frac{1}{n}\vol(\lambda)^{\frac{1-n}{n}}\Sa(K_\lambda;E).
\end{equation}
\end{lem}

The following statement corresponds to formula (18) of \cite[Section 6]{Ha55} if $E=B_n$. Although the proof is analogous to the proof in the case  $E=B_n$, we include it for the sake of completeness.

\begin{lem}
\label{lem: in of out}
For all $K,E \in \mathcal{K}_n^n$ and $-\inr(K;E) < \lambda_0 < \lambda_1 < \infty$,
$$
K_{\lambda_0}=K_{\lambda_1} \sim |\lambda_0-\lambda_1|E.
$$
\end{lem}

\begin{proof}
Let $K,E\in\K^n_n$ and $-\inr(K;E) < \lambda_0 < \lambda_1$.

Equation \eqref{eq:  Kl intersection} implies directly $K_{\lambda_1} \sim |\lambda_0-\lambda_1|E \subseteq K_{\lambda_0}$. 

For the converse, let $x \in K_{\lambda_0}$. Then, by \eqref{eq:  Kl intersection}, $x  \in \bigcap_{u \in \s^{n-1}} H^-_{u,h_K(u)+\lambda_0 h_E(u)}$. Thus,
\[
x+|\lambda_0-\lambda_1|E=x+(\lambda_1-\lambda_0)E \subseteq
\bigcap_{u \in \s^{n-1}} H^-_{u,(h_K(u)+\lambda_0 h_E(u))+(\lambda_1-\lambda_0)h_E(u)}
=K_{\lambda_1},
\]
whence $x \in K_{\lambda_1} \sim |\lambda_0-\lambda_1|E$.
\end{proof}

\begin{lem}
\label{lem:  K+E K}
Let $K,L \in \mathcal{K}_n^n$. If $K+L$ is homothetic to $L$ then $K$ is homothetic to $L$.
\end{lem}

\begin{proof}
Let $K,L \in \mathcal{K}_n^n$, and let $K+L$ be homothetic to $L$. Then, there are $\alpha > 1$ and $x_0 \in \mathbb{R}^n$ such that $K+L=\alpha L + x_0$. Hence, rewriting the latter in terms of  support functions 
$$
h(K,\cdot)+h(L,\cdot)=h(K+L,\cdot)=h(\alpha L+x_0,\cdot)=\alpha h(L,\cdot)+\langle x_0,\cdot\rangle,
$$
this yields
$$
h(K,\cdot)=(\alpha-1)h(L,\cdot)+\langle x_0,\cdot\rangle=h((\alpha-1)L+x_0,\cdot)
$$
and in turn $K=(\alpha-1)L+x_0$.
\end{proof}

Now we can prove the monotonicity of the isoperimetric quotient of the family $(K_\lambda)_{\lambda > -\inr(K;E)}$.

\begin{thm}
\label{thm:overline{I}}
Let $K,E\in \mathcal{K}^n_n$ be convex bodies. Then the function
\[
\I(\lambda)=\frac{\Sa(K_\lambda;E)^n}{\vol(K_\lambda)^{n-1}}
\]
is monotonically decreasing on $(-\inr(K;E),\infty)$.

Moreover, the following are equivalent for all $-\inr(K;E) < \lambda_0 < \lambda_1 < \infty$:
\begin{itemize}\itemsep5pt
\item[(i)]
$\I(\lambda_0)=\I(\lambda_1)$,
\item[(ii)]
$K_{\lambda_0}$ is homothetic to $K_{\lambda_1}$,
\item[(iii)]
$K_{\lambda_1}$ is homothetic to a tangential body of $E$,
\item[(iv)]
$\I(\lambda)$ is constant on $(-\inr(K;E),\lambda_1]$.
\end{itemize}

If $\lambda_1 > 0$, the equivalent conditions (i)-(iv) are satisfied if and only if $K$ is homothetic to $E$ and, consequently, if and only if $\I(\lambda)=n^n\vol(E)$
for all $\lambda \in (-\inr(K;E),\infty)$. 
\end{thm}

\begin{proof}
Let $K,E\in \mathcal{K}^n_n$. Then, the case $i=0$ in Theorem \ref{gral BM th} and Lemma~\ref{fact:parallel_concave} ensure that $\lambda \mapsto \vol(\lambda)^{\frac{1}{n}}$ defines a concave function on $(-\inr(K;E),\infty)$. Further, by Lemma~\ref{l:der vol^1/n}, the derivative $\frac{\dlat}{\dlat \lambda} \left(\vol(\lambda)^{\frac{1}{n}}\right)$ of that function exists, and is monotonically decreasing (see e.g.\ \cite{Royden}).
Hence, by \eqref{e:der vol^1/n},
$\I(\lambda)=\left(n\frac{\dlat}{\dlat \lambda} \left(\vol(\lambda)^{\frac{1}{n}}\right)\right)^n$ decreases as well.

(i)$\Rightarrow$(ii). Condition (i), together with the just proven monotonicity of the function $\I(\lambda)$, implies that $\I(\lambda)=\left(n\frac{\dlat}{\dlat \lambda} \left(\vol(\lambda)^{\frac{1}{n}}\right)\right)^n$ is constant on $[\lambda_0,\lambda_1]$. Thus, the function $\lambda\mapsto \vol(\lambda)^{\frac{1}{n}}$ is affine on $[\lambda_0,\lambda_1]$, and we have equality in \eqref{BM ineq quermass} for $i=0$ and all $\mu \in [0,1]$. Now, by Theorem~\ref{gral BM th}, $K_{\lambda_0}$ and $K_{\lambda_1}$ are homothetic.

(ii)$\Rightarrow$(iii).  Let $K_{\lambda_0}$ and $K_{\lambda_1}$ be homothetic. Then, by Lemma~\ref{lem: in of out}, we can ensure that the inner parallel body $K_{\lambda_0}=K_{\lambda_1} \sim |\lambda_1-\lambda_0|E$ of $K_{\lambda_1}$ is homothetic to $K_{\lambda_1}$. Now Theorem~\ref{th:  Sch tang} shows that $K_{\lambda_1}$ is homothetic to a tangential body of $E$.

(iii)$\Rightarrow$(iv). Let $\lambda \in (-\inr(K;E),\lambda_1)$, and $K_{\lambda_1}$ be homothetic to a tangential body of $E$. Then, by Lemma~\ref{lem:  in of out}, $K_\lambda=K_{\lambda_1} \sim |\lambda-\lambda_1|E$ is an inner parallel body of $K_{\lambda_1}$. Thus, Theorem~\ref{th:  Sch tang} together with condition (iii), i.e., $K_{\lambda_1}$ is homothetic to a tangential body of $E$,  implies that $K_\lambda$ is homothetic to $K_{\lambda_1}$. Since the isoperimetric quotient $K\mapsto\I(K;E)$ agrees for homothetic bodies, $\I(\lambda)=\I(\lambda_1)$.

(iv)$\Rightarrow$(i) is trivial.

Now we deal with the last claim. For that, let $\lambda_1 >0$, and assume that any of the equivalent assertions (i)-(iv) holds. By (iv), $\I(0)=\I(\lambda_1)$, since $-\inr(K;E) < 0 < \lambda_1$. Then, by (ii), $K_0=K$ is homothetic to $K_{\lambda_1}=K+\lambda_1 E$, which together with Lemma~\ref{lem:  K+E K} yields that $K$ is homothetic to $\lambda_1 E$ and thus $K$ is homothetic to $E$ itself.

Since the isoperimetric quotient $K\mapsto I(K;E)$ is invariant under homotheties, w.l.o.g. we can assume that $K=E$. Then, $K_\lambda=(1+\lambda)E$ for $-\inr(K;E)=-1 < \lambda$. Again, by the mentioned invariance, we obtain $\I(\lambda)=\I(E;E)$ for all $-\inr(K;E) < \lambda$. Finally, $\I(E;E)=n^n\vol(E)$, using that $\W_i(E;E)=\vol(E)$ for all $0\leq i\leq n$ together with \eqref{e: rel Sa}.
\end{proof}

As a corollary, we obtain that the classical isoperimetric quotient is monotonically decreasing, by setting $E=B_n$.

\begin{cor}\label{c: class isop quot}
For every $K \in \mathcal{K}_n^n$ and $E=B_n$, the function 
\[
\I(\lambda)=\frac{\Sa(K_\lambda)^n}{\vol(K_\lambda)^{n-1}}
\]
is monotonically decreasing on $(-\inr(K;B_n),\infty)$.
\end{cor}

\begin{rem}
It is crucial for the monotonicity in Theorem~\ref{thm:overline{I}} that the surface area is considered relative to the gauge body $E$: For example, consider $K=[0,1]^2$ and $E=[0,1] \times [0,2]$ in $\mathbb{R}^2$. Then $K_0=K$, $K_1=[0,2]\times [0,3]$, and the classical isoperimetric quotients, where the surface area is taken relative to $B_2$, are
$\I(K_0)=\frac{4^2}{1}=16 < \frac{50}{3}=\frac{10^2}{6}=\I(K_1)$.
\end{rem}


\section{Monotonicity for related families of bodies and isoperimetric problems}\label{sec:  Wulff}

In this section we introduce some new (one-parameter) families of bodies related to the family $(K_\lambda)_{\lambda \ge -\inr(K;E)}$ of parallel bodies. First we recall that a body $K$ is \emph{determined by} a set $\Omega \subseteq \s^{n-1}$ if 
$$
K = \bigcap_{u \in \Omega} H^-_{u,h_K(u)}
$$
\cite[pp.\ 385, 411]{Sch}.

\begin{rem}[{\cite[p.\ 386]{Sch}}] \label{KintersectionExtreme} 
The smallest closed set $\Omega\subseteq\s^{n-1}$ that determines a given convex body $K$ is the closure of the set $\U(K)$ of outer unit normal vectors at regular boundary points of $K$. In other words, $\Omega \subseteq \s^{n-1}$ determines $K$ if and only if $\U(K) \subseteq \cl(\Omega)$, where $\cl(\cdot)$ denotes the closure operator.

The elements of $\U(K)$ are also known in the literature as extreme normal vectors of $K$. The set $\U(K)$ need not be closed; e.g.\ when $K \in \K_2^2$ is a semicircle.
\end{rem}

First we prove that in the definition \eqref{eq:  Kl intersection} of inner parallel bodies of $K\in\K^n_n$ relative to $E\in\K^n_n$, we can replace the complete sphere $\s^{n-1}$ in the intersection by any subset $\Omega$ that determines $K$.
\begin{lem}
\label{lem: inner representation}
Let $K,E \in \mathcal{K}_n^n$ and let $\Omega \subseteq \s^{n-1}$ determine $K$. Then
\begin{equation}
\label{eq:  inner representation}
K_\lambda=\bigcap_{u \in \Omega} H^-_{u,h_K(u)+\lambda h_E(u)}
\end{equation}
for all $\lambda \in [-\inr(K;E),0]$.
\end{lem}

\begin{proof}
Let $K,E \in \mathcal{K}_n^n$ and let $\Omega \subseteq \s^{n-1}$ determine $K$.

By \eqref{eq:  Kl intersection}, it is evident that $K_\lambda\subseteq\bigcap_{u \in \Omega} H^-_{u,h_K(u)+\lambda h_E(u)}$. 

For the verification of the reverse inclusion, let $x\in\bigcap_{u \in \Omega} H^-_{u,h_K(u)+\lambda h_E(u)} $ be arbitrary. Then, $x \in H^-_{u,h_K(u)-|\lambda|h_E(u)}$ for all $u \in \Omega$, i.e., $\langle x, u\rangle  \leq h_K(u)-|\lambda|h_E(u)$ for all $u \in \Omega$. We observe that the latter yields that $x+|\lambda|E \subseteq H^-_{u,h_K(u)}$ for every $u\in\Omega$, since clearly, for any $e\in E$ and $u\in\Omega$,  $\langle x+\lambda e, u\rangle  \leq h_K(u)$.

Consequently, 
\[x+|\lambda|E \subseteq \bigcap_{u \in \Omega} H^-_{u,h_K(u)}=K,
\] where the last identity follows from the fact that $\Omega$ determines $K$. Hence, $x \in K_\lambda$.
\end{proof}

\begin{rem}

Given a set $\Omega \subseteq \s^{n-1}$ that contains the origin in  the interior of
its convex hull, we define the set
\[
E^\Omega= \bigcap_{u \in \Omega} H^-_{u,h_E(u)}.
\] 
Observing that $E\subseteq E^\Omega$ and that $h_E(u)=h_{E^\Omega}(u)$ for $u \in \Omega$, it is clear that $E^\Omega$ is a tangential body of $E$. Indeed, it follows from the definition that $E^\Omega$ is the smallest tangential body of $E$ that is determined by $\Omega$.
\end{rem}

As a particular case of the latter we can obtain the so-called \emph{form body of $K$ relative to $E$}, for any $K \in \mathcal{K}_n^n$. This is defined (cf.\ \cite[p.\ 386]{Sch}) as
\[
K^\ast= E\,^{\U(K)}.
\] 
We notice that $\U(K^\ast)\subseteq\cl(\U(K))$ and that the inclusion may be strict \cite[Lemma 4.6 and the preceding example]{SYphd}.

For fixed $K,E\in\K^n_n$ and $\Omega\subseteq\s^{n-1}$ determining $K$, Lemma~\ref{lem:  inner representation} motivates the introduction of the following one-parameter family of convex bodies associated to $K$,
\begin{equation}\label{eq:  def KlO}
K(\Omega,\lambda)=\bigcap_{u \in \Omega} H^-_{u,h_K(u)+\lambda h_E(u)}, \qquad \lambda \in [-\inr(K;E), \infty).
\end{equation}
Here the representation \eqref{eq:  inner representation} is naturally extended to $\lambda > 0$. 
The sets $K(\Omega,\lambda)$, $\lambda\in[-\inr(K;E),\infty)$, are, according to \cite[Section 7.5]{Sch}, Wulff shapes or Alexandrov bodies associated with the pairs $(\Omega,f_\lambda)$ where $f_\lambda:\Omega\to \mathbb{R},u\mapsto h_K(u)+\lambda h_E(u)$.

Our aim is to prove that the relative isoperimetric quotient functions defined for these new one-parameter families of bodies are also decreasing. Before turning to that, we prove the following lemma, which will be useful for later considerations.

\begin{lem}\label{lem:  rem inner}
Let $K,E\in\K^n_n$ and let $\Omega\subseteq \s^{n-1}$ determine $K$, i.e., $\U(K) \subseteq \cl(\Omega)$. Then,

\begin{itemize}\itemsep5pt

\item[(i)] $K_\lambda=K(\s^{n-1},\lambda)$ for all $\lambda\in[-\inr(K;E),\infty)$,
\item[(ii)] $K_\lambda = K \sim |\lambda| E = K \sim |\lambda| E^\Omega = K \sim |\lambda| K^\ast$ for all $\lambda \in [-\inr(K;E),0]$,
\item[(iii)] $K_\lambda=K(\Omega,\lambda)$ for all $\lambda \in [-\inr(K;E),0]$,
\item[(iv)] $K+\lambda E=K_\lambda\subseteq K(\Omega,\lambda)$ for all $\lambda\geq 0$, and the inclusion may be strict.
\end{itemize}
\end{lem}

\begin{proof}
Let $K,E\in\K^n_n$, and let $\Omega\subseteq \s^{n-1}$ so that $\U(K) \subseteq \cl(\Omega)$.

(i) follows directly from \eqref{eq:  Kl intersection}.

(ii): The first equality is the definition of $K_\lambda$. The second one follows from Lemma~\ref{lem:  inner representation} and the observation that $h_E(u)=h_{E^\Omega}(u)$ for all $u \in \Omega$. Putting $\Omega=\mathcal{U}(K)$ as a particular case, we obtain the last equality.

(iii) is Lemma~\ref{lem:  inner representation}.

(iv) follows from (i). To prove that the inclusion may be strict, it is enough to consider in the plane $E=B_2$, $K=[0,1]^2$ a square, and $\Omega=\mathcal{U}(K)=\{(\pm 1,0),(0,\pm 1)\}$. Then for any $\lambda>0$, the outer parallel body $K+\lambda B_2$ is strictly contained in the square $K(\Omega,\lambda)$. 
\end{proof}

Next we state the result about the relative isoperimetric quotient for the family $K(\Omega,\lambda)$ which we aim to prove.

\begin{thm}
\label{thm:  Wulff}
Let $K,E \in \mathcal{K}_n^n$ and let $\Omega \subseteq \s^{n-1}$ determine $K$. Then the relative isoperimetric quotient function
$$
\I^\Omega(\lambda)=\frac{\Sa(K(\Omega,\lambda);E)^n}{\vol(K(\Omega,\lambda))^{n-1}}
$$
of the family $(K(\Omega,\lambda))_{\lambda > -\inr(K;E)}$ is monotonically decreasing on $(-\inr(K;E),\infty)$.

Moreover, the following are equivalent for all $-\inr(K;E) < \lambda_0 < \lambda_1 < \infty$:
\begin{itemize}\itemsep5pt
\item[(i)]
$\I^\Omega(\lambda_0)=\I^\Omega(\lambda_1)$,
\item[(ii)]
$K(\Omega,\lambda_0)$ is homothetic to $K(\Omega,\lambda_1)$,
\item[(iii)]
$K(\Omega,\lambda_1)$ is homothetic to a tangential body of $E$,
\item[(iv)]
$\I^\Omega(\lambda)$ is constant on $(-\inr(K;E),\lambda_1]$.
\end{itemize}

If $\lambda_1 > 0$, the equivalent conditions (i)-(iv) are satisfied if and only if $K$ is homothetic to $E^\Omega$ and, consequently, if and only if $\I(\lambda)= \I\left(E^\Omega;E\right)=n^n \vol\left(E^\Omega\right)$ for all $\lambda \in (-\inr(K;E),\infty)$. 
\end{thm}

The proof of Theorem~\ref{thm:  Wulff} is based on the following lemmas.

\begin{lem}
\label{lem:  Wulff}
Let $K,E \in \mathcal{K}_n^n$ and let $\Omega \subseteq \s^{n-1}$ determine $K$. If $\Lambda > -\inr(K;E)$, then
\[
K(\Omega,\lambda)=(K(\Omega,\Lambda))_{\lambda-\Lambda},\quad \text{for all}\quad\lambda \in (-\inr(K;E),\Lambda].
\]
\end{lem}

\begin{proof}
Let $K,E \in \mathcal{K}_n^n$ and let $\Omega \subseteq \s^{n-1}$ determine $K$. 

From the definition of $K(\Omega,\Lambda)$ it follows that $K(\Omega,\Lambda)$ is determined by $\Omega$. It follows also from the definition, that $h_{K(\Omega,\Lambda)}(u) \le h_K(u)+\Lambda h_E(u)$ for all $u \in \Omega$. Using the latter, together with Lemma~\ref{lem:  inner representation}, we obtain
\[
(K(\Omega,\Lambda))_{\lambda-\Lambda}=\bigcap_{u \in \Omega} H^-_{u,h_{K(\Omega,\Lambda)}(u)+(\lambda-\Lambda)h_E(u)} \subseteq \bigcap_{u \in \Omega} H^-_{u,(h_K(u)+\Lambda h_E(u))+(\lambda-\Lambda)h_E(u)} =K(\Omega,\lambda).
\]

To prove the reverse inclusion, let $x \in K(\Omega,\lambda)$. Then $x \in H^-_{u,h_K(u)+\lambda h_E(u)}$ for every $u \in \Omega$. As in the proof of  Lemma \ref{lem: inner representation}, it follows that
\[
x+|\lambda-\Lambda|E \subseteq H^-_{u,(h_K(u)+\lambda h_E(u))+|\lambda-\Lambda|h_E(u)}=H^-_{u,h_K(u)+\Lambda h_E(u)}
\]
for every $u \in \Omega$. Since the last inclusion holds for all $u \in \Omega$, we get $x+|\lambda-\Lambda|E \subseteq K(\Omega,\Lambda)$. That is, $x \in (K(\Omega,\lambda))_{\lambda-\Lambda}$.
\end{proof}

\begin{lem}\label{lem: h K(Omega,lambda) on Omega}
Let $K,E \in \mathcal{K}_n^n$ and let $\Omega \subseteq \s^{n-1}$ determine $K$. Then, for any $\lambda > 0$,
\begin{equation}\label{eq:  hKO}
h_{K(\Omega,\lambda)}(u)= h_K(u)+\lambda h_E(u) \qquad \text{for all } u\in \Omega.
\end{equation}
\end{lem}

\begin{proof}
Let $K,E \in \mathcal{K}_n^n$ and let $\Omega \subseteq \s^{n-1}$ determine $K$. 
Directly from the definition \eqref{eq:  def KlO} of $K(\Omega,\lambda)$, it follows that $h_{K(\Omega,\lambda)}(u) \le h_K(u)+\lambda h_E(u)$ for all $u \in \Omega$.

 On the other hand, we have 
\[
K(\Omega,\lambda)=\bigcap_{u \in \Omega} H^-_{u,h_K(u)+\lambda h_E(u)} \supseteq \bigcap_{u \in \s^{n-1}} H^-_{u,h_K(u)+\lambda h_E(u)}=K+\lambda E.
\]
From the latter it follows that $h_{K(\Omega,\lambda)}(u) \ge h_K(u)+\lambda h_E(u)$ for any $u\in\s^{n-1}$. Thus,
\begin{equation*}
h_{K(\Omega,\lambda)}(u)= h_K(u)+\lambda h_E(u) \qquad \text{for all }u \in \Omega.
\end{equation*}
\end{proof}

\begin{lem}\label{lem:  hKO}
Let $K,E \in \mathcal{K}_n^n$, let $\Omega \subseteq \s^{n-1}$ determine $K$, and let $\lambda > 0$. If $K(\Omega,\lambda)$ is homothetic to $K$ then $K$ is homothetic to $E^\Omega$.
\end{lem}

\begin{proof}
Let $K,E \in \mathcal{K}_n^n$ and let $\Omega \subseteq \s^{n-1}$ determine $K$. Let $\lambda > 0$ and assume that $K(\Omega,\lambda)$ is homothetic to $K$. Then, there are $\alpha > 1$ and $x_0 \in \mathbb{R}^n$ such that $K(\Omega,\lambda)=\alpha K + x_0$. Hence, $h_{K(\Omega,\lambda)}=\alpha h_K+ \langle x_0,\cdot\rangle$. Thus, \eqref{eq:  hKO} yields
\[
h_E(u)= \frac{1}{\lambda}\bigr( (\alpha-1)h_K(u)+\langle x_0,u\rangle\bigl), \qquad u \in \Omega.
\]
Finally, it is enough to observe that
\begin{align*}
E^\Omega &= \bigcap_{u \in \Omega} H^-_{u,h_E(u)}
= \bigcap_{u \in \Omega} H^-_{u,\frac{1}{\lambda}( (\alpha-1)h_K(u)+\langle x_0,u\rangle)}\\
&= \frac{1}{\lambda}\left( (\alpha-1)\bigcap_{u \in \Omega} H^-_{u,h_K(u)}+ x_0\right)=\frac{1}{\lambda}\bigr((\alpha-1)K+x_0\bigr).
\end{align*}
That is, $K$ is homothetic to $E^\Omega$.
\end{proof}

Now we can prove Theorem~\ref{thm:  Wulff}

\begin{proof}[Proof of Theorem~\ref{thm:  Wulff}]
Let $K,E \in \mathcal{K}_n^n$ and let $\Omega \subseteq \s^{n-1}$ determine $K$. In order to prove that the relative isoperimetric quotient function $\I^\Omega(\lambda)=\frac{\Sa(K(\Omega,\lambda);E)^n}{\vol(K(\Omega,\lambda))^{n-1}}$ is decreasing, we will interpret the assertion in terms of the relative (to $E$) isoperimetric function $\I(\lambda)=\I(L_\lambda;E)=\frac{\Sa(L_\lambda;E)^{n}}{\vol(L_\lambda)^{n-1}}$ for an appropriate convex body $L$, since for this function we have already proven in Theorem \ref{thm:overline{I}}, that it is decreasing.

Let $\Lambda> -\inr(K;E)$ be fixed. By Lemma~\ref{lem:  Wulff}, $K(\Omega,\lambda)=(K(\Omega,\Lambda))_{\lambda-\Lambda}$ for any $\lambda\in(-\inr(K;E),\Lambda]$ and thus, $\I^\Omega(\lambda)=\I(K(\Omega,\Lambda)_{\lambda-\Lambda};E)$ for all $\lambda \in (-\inr(K;E),\Lambda]$. Hence, a direct application of Theorem~\ref{thm:overline{I}} to the convex body $K(\Omega,\Lambda)$ gives the claim of Theorem~\ref{thm:  Wulff} for all $\lambda,\lambda_0,\lambda_1 \in (-\inr(K;E),\Lambda]$.

Since this can be done for all $\Lambda > -\inr(K;E)$, the proof of the monotonicity and the equivalence of (i)-(iv) is complete. It only remains to prove the claim concerning $\lambda_1 >0$.

Let $\lambda_1>0$, and assume that any of the equivalent assertions (i)-(iv) holds. Using (iv) we have $\I^\Omega(0)=\I^\Omega(\lambda_1)$, since $-\inr(K;E) < 0 < \lambda_1$. Now, (ii) shows that $K(\Omega,0)=K$ is homothetic to $K(\Omega,\lambda_1)$, which together with Lemma~\ref{lem:  hKO} yields that $K$ is homothetic to $E^\Omega$.

Finally, if $K$ is homothetic to $E^\Omega$, the invariance under homotheties of the isoperimetric quotient provides us  with the last assertion. Indeed, assume that $K=E^\Omega=\bigcap_{u \in \Omega} H^-_{u,h_E(u)}$ w.l.o.g., then $K(\Omega,\lambda)=
(1+\lambda)E^\Omega$ and $\I^\Omega(\lambda)=\I^\Omega\left((1+\lambda)E^\Omega;E\right)=\I\left(E^\Omega;E\right)$ for $-\inr(K;E)=-1 < \lambda$. Finally, since $E^\Omega$ is a tangential body of $E$, \cite[Theorem 7.6.17]{Sch} gives $\W_0\left(E^\Omega;E\right)=\W_1\left(E^\Omega;E\right)$; i.e., $\vol\left(E^\Omega\right)=\frac{1}{n}\Sa\left(E^\Omega;E\right)$. Thus, $\I\left(E^\Omega;E\right)=n^n \vol\left(E^\Omega\right)$, which finishes the proof.
\end{proof}

\begin{rem}
We have obtained Theorem~\ref{thm:  Wulff} as a consequence of Theorem~\ref{thm:overline{I}}. Conversely, Theorem~\ref{thm:overline{I}} follows from Theorem~\ref{thm:  Wulff}, because
$K_\lambda=K\left(\s^{n-1},\lambda\right)$ by Lemma~\ref{lem:  rem inner}(i),
and in turn $\I(\lambda)=\I^{\s^{n-1}}(\lambda)$.
\end{rem}

Let us have a look at relations between the classes $(K_\lambda)_{\lambda \ge -\inr(K;E)}$ and $(K(\Omega,\lambda))_{\lambda \ge -\inr(K;E)}$.

\begin{prop}\label{prop:  K_Omega tang}Let $K,E \in \mathcal{K}_n^n$, let $\Omega_1,\Omega_2 \subseteq \s^{n-1}$ both determine $K$, and suppose that $\Omega_1 \subseteq \Omega_2$. Then
\begin{equation}\label{eq:  Omega12}
K(\Omega_1,\lambda) \supseteq K(\Omega_2,\lambda) \supseteq K_\lambda
\end{equation}
for all $\lambda \ge -\inr(K;E)$.
Furthermore,
\begin{enumerate}
\item for $\lambda \le 0$, there is equality all over \eqref{eq:  Omega12};
\item if $\lambda > 0$, then both $K(\Omega_1,\lambda)$ and $K(\Omega_2,\lambda)$ are tangential bodies of $K_\lambda$, and moreover, $K(\Omega_1,\lambda)$ is a tangential body of $K(\Omega_2,\lambda)$.
\end{enumerate}
\end{prop}

\begin{proof}
Let $K,E \in \mathcal{K}_n^n$, and let $\Omega_1,\Omega_2 \subseteq \s^{n-1}$ both determine $K$. 

The relations \eqref{eq:  def KlO} and \eqref{eq:  Kl intersection} directly yield \eqref{eq:  Omega12}. 

(i) follows from Lemma~\ref{lem:  inner representation}.

(ii): Let $\lambda\geq 0$. Then Lemma~\ref{lem: h K(Omega,lambda) on Omega} shows that $h_{K(\Omega_1,\lambda)}(u)=h_{K_\lambda}(u)$ for all $u \in \Omega_1$. Consequently, $K(\Omega_1,\lambda)$ is a tangential body of $K_\lambda$, since $K(\Omega_1,\lambda)$ is determined by $\Omega_1$ according to its definition \eqref{eq:  def KlO}. By the same arguments, $K(\Omega_2,\lambda)$ is a tangential body of $K_\lambda$.

Finally, since $K(\Omega_1,\lambda)$ is a tangential body of $K_\lambda$, inclusions \eqref{eq:  Omega12} imply that $K(\Omega_1,\lambda)$ is also a tangential body of $K(\Omega_2,\lambda)$, because every common supporting hyperplane of $K(\Omega_1,\lambda)$ and $K_\lambda$ supports $K(\Omega_2,\lambda)$, too.
\end{proof} 

Next we consider the asymptotic behavior of $K_\lambda$ and $K(\Omega,\lambda)$ for $\lambda \to \infty$. We obtain in particular that the inclusions in \eqref{eq:  Omega12} may be strict for large $\lambda$, since the respective limit shapes can be different.

\begin{prop}\label{prop:  classes}
Let $K,E \in \mathcal{K}_n^n$ and let $\Omega \subseteq \s^{n-1}$ determine $K$. Then the convex bodies $K_\lambda$ and $K(\Omega,\lambda)$ converge {\it in $\lambda$} in the sense that
\[
\lim_{\lambda\to\infty} \frac{1}{\lambda}K_\lambda=E
\quad\text{ and }\quad
\lim_{\lambda\to\infty} \frac{1}{\lambda}K(\lambda,\Omega)=E^\Omega
\]
in the Hausdorff metric. Moreover,
\[
\lim_{\lambda\to\infty} \I(\lambda)=\I(E;E)=n^n \vol(E)
\quad\text{ and }\quad
\lim_{\lambda\to\infty} \I^\Omega(\lambda)=\I\left(E^\Omega;E\right)=n^n \vol\left(E^\Omega\right).
\]
\end{prop}
We remark that the convergence result and the limit value for $K_\lambda$ can, in its essence, be found in \cite[p. 56]{SYphd}.

\begin{proof}
The convergences of the bodies follow from their definitions, i.e., from \eqref{eq:  Kl intersection} and \eqref{eq:  def KlO}. 
The remaining claims $\I(E;E)=n^n\vol(E)$ and $\I\left(E^\Omega;E\right)=n^n\vol\left(E^\Omega\right)$ have been shown in the proofs of Theorems~\ref{thm:overline{I}} and \ref{thm:  Wulff}, respectively.\end{proof}

\begin{rem}
Suppose that $\Omega$ determines both $K$ and $E$. Then the families $(K_\lambda)_{\lambda \ge -\inr(K;E)}$ and $(K(\Omega,\lambda))_{\lambda \ge -\inr(K;E)}$ agree for $\lambda \le 0$ and have the same limit shape $E=E^\Omega$. Nevertheless, they do not necessarily coincide. In fact, \cite[pp.\ 23--24]{SYphd} and \cite[Section~3]{HCSq} give examples of polytopes $K,E \in \mathcal{K}_3^3$ such that $\U(K)=\U(E) \subsetneq  \U(K+E)$. Then $K_1 \ne K(\U(K),1)$, because $K(\U(K),1)$ is determined by $\U(K)$ whereas $K_1=K+E$ is not.
\end{rem}

Next we pose a natural question in this context, which we have not been able to answer so far.
\begin{quest}
Let $K,E \in \mathcal{K}_n^n$, let $\Omega_1,\Omega_2  \subseteq \s^{n-1}$ both determine $K$ and suppose that $\Omega_1 \subseteq \Omega_2$. We know that $\I^{\Omega_1}(\lambda)=\I^{\Omega_2}(\lambda)$ for $-\inr(K;E) < \lambda \le 0$ and that $\lim_{\lambda \to \infty} \I^{\Omega_1}(\lambda) \ge \lim_{\lambda \to \infty} \I^{\Omega_2}(\lambda)$. Do we have $\I^{\Omega_1}(\lambda) \ge \I^{\Omega_2}(\lambda)$ also for all $\lambda > 0$?
\end{quest}

We come to isoperimetrically optimal bodies.

\begin{cor}\label{cor:  isop optimal}
Let $E \in \mathcal{K}_n^n$ and let $\Omega \subseteq \s^{n-1}$ be a set that contains the origin in the interior of its convex hull.

Then a convex body $\tilde{K} \in \mathcal{K}_n^n$ is a minimizer of the relative isoperimetric quotient $\I(K;E)$ among all convex bodies $K \in \mathcal{K}_n^n$ that are determined by $\Omega$ if and only if $\tilde{K}$ is homothetic to the tangential body $E^{\Omega}$ of $E$. In particular, that minimal quotient is $\I\left(E^\Omega;E\right)=n^n \vol\left(E^\Omega\right)$.
\end{cor}

\begin{proof}
Let $K \in \mathcal{K}_n^n$ be determined by $\Omega$. By Theorem~\ref{thm:  Wulff} and Proposition~\ref{prop:  classes}, 
\[
\I(K;E)=\I^\Omega(0) \ge \lim_{\lambda \to \infty} \I^\Omega(\lambda)=\I\left(E^\Omega;E\right)=n^n \vol\left(E^\Omega\right)
\]
with equality if and only if $\I^\Omega(\lambda)$ is constant for $\lambda \ge 0$. The latter holds if and only if all convex bodies $K(\Omega,\lambda)$, $\lambda \ge 0$, are mutually homothetic, and in turn are homothetic to $\lim_{\lambda \to \infty} \frac{1}{\lambda} K(\Omega,\lambda)=E^\Omega$.
\end{proof}

\begin{rem}
If $E=B_n$, Corollary~\ref{cor:  isop optimal} concerns the classical isoperimetric quotient from \eqref{e:class isop quot}. For that case Corollary~\ref{cor:  isop optimal} is a well-known result \cite[p.~385]{Sch}, that goes back to Lindel\"of and Minkowski \cite{Lindeloef1869, Minkowski1911} 
for finite $\Omega$ and to Aleksandrov \cite{Alexandrov1938} for general $\Omega$.

For $E=B_n$ and $\Omega=\s^{n-1}$, we obtain the isoperimetric inequality for arbitrary convex bodies.
\end{rem}


\section{Isoperimetric-type quotients of quermassintegrals}\label{sec:  other qmi}

This section is motivated by the following monotonicity result from \cite{HCSq}.

\begin{prop}[{\cite[Proposition 4.3, Remark 4.4]{HCSq}}]\label{prop:  HCSq}
Let $K,E \in \mathcal{K}_n^n$, let $0 \le i < j < n$, and suppose that $K$ belongs to the class $\r_j$. Then the function
\[
\lambda\mapsto \frac{\W_j(K_\lambda;E)^{n-i}}{\W_i(K_\lambda;E)^{n-j}}
\]
is monotonically decreasing on $(-\inr(K;E),\infty)$.
\end{prop}

Here the classes $\r_j$ are defined by a differentiability condition of the functions $\lambda\mapsto \W_i(K_\lambda;E)$ on $[-\inr(K;E),\infty)$.
\begin{defn}[\cite{HCS}]\label{d:r_p}    
Let $E\in\K^n_n$ and let $p$ be an integer, $0\leq p\leq n-1$. A
convex body $K\in\K^n$ belongs to the {\em class $\r_p$} if, for
all $0\leq i\leq p$ and for $-\inr(K;E)\leq\lambda<\infty$, the following equalities
hold
\begin{equation*}
\frac{\dlat^-}{\dlat\lambda}\W_i(\lambda)=\frac{\dlat^+}{\dlat\lambda}\W_i(\lambda)=(n-i)\W_{i+1}(\lambda),
\end{equation*}
where the first equation is to be dropped when $\lambda=-\inr(K;E)$.
\end{defn}

We remark that the above definition is natural taking \eqref{e:concav Wi} into account. Proposition~\ref{p:  diff vol} yields
\[
\r_0 = \K^n 
\]
for any $E\in\K^n_n$. Clearly,
\begin{equation}\label{e: content classes ri}
\r_{n-1} \subseteq \r_{n-2} \subseteq \ldots \subseteq \r_1 \subseteq \r_0,
\end{equation}
and the inclusions are strict in general, as was shown in \cite{HCS}.

The case $i=0$, $j=1$ in Proposition~\ref{prop: HCSq} gives the monotonicity of $\I(\lambda)$ proven in Theorem~\ref{thm:overline{I}}, but only under the additional assumption $K \in \r_1$. In a similar way as we have relaxed the assumptions (to none) in the case $i=0$, $j=1$ in Theorem \ref{thm:overline{I}}, in the next we shall see that, more generally, the assumption $K \in \r_j$ from Proposition~\ref{prop: HCSq} can be relaxed by $K \in \r_{j-1}$.
We need the following lemma.

\begin{lem}\label{l:der Wi^1/n-i}
Let $0\leq i\leq n-1$, $E \in \mathcal{K}_n^n$, $K \in \r_i$ and $-\inr(K;E) < \lambda < \infty$. Then
\begin{equation}\label{e:der Wi^1/n-i}
\frac{\dlat}{\dlat \lambda} \left(\W_i(\lambda)^{\frac{1}{n-i}}\right)=\W_i(\lambda)^{\frac{1-n+i}{n-i}}\W_{i+1}(\lambda).
\end{equation}
\end{lem}

\begin{proof}
The result follows immediately from the definition of the class $\r_i$ and the standard rules of differentiation.
\end{proof}

\begin{prop}\label{p:i,i+1 quot}

Let $0\leq i\leq n-2$, $K,E \in \mathcal{K}_n^n$ and $K \in \r_i$.
Then the function
\[
\I_i(\lambda)=\frac{\W_{i+1}(K_{\lambda};E)^{n-i}}{\W_i(K_{\lambda};E)^{n-i-1}}
\]
is monotonically decreasing on $(-\inr(K;E),\infty)$.

Moreover, if $E$ is smooth, the following are equivalent for all $-\inr(K;E) < \lambda_0 < \lambda_1 < \infty$:
\begin{itemize}\itemsep5pt
\item[(i)]
$\I_i(\lambda_0)=\I_i(\lambda_1)$,
\item[(ii)]
$K_{\lambda_0}$ is homothetic to $K_{\lambda_1}$,
\item[(iii)]
$K_{\lambda_1}$ is homothetic to an $(n-i-1)$-tangential body of $E$,
\item[(iv)]
$\I_i(\lambda)$ is constant on $(-\inr(K;E),\lambda_1]$.
\end{itemize}

If $E$ is smooth and $\lambda_1 > 0$, conditions (i)-(iv) are satisfied if and only if $K$ is homothetic to $E$ and, consequently, if and only if $\I_i(\lambda)=\vol(E)$ for all $\lambda \in (-\inr(K;E),\infty)$.
\end{prop}

The case $i=n-1$ is excluded, since $\I_{n-1}(\lambda)=
\W_n(K_\lambda;E)=\vol(E)$ is constant.

Before dealing with the proof, we introduce a refined definition of tangential bodies (of a fixed convex body), for which we need the notion of $p$-extreme supporting hyperplanes, $0\leq p\leq n-1$. Given a convex body $K\in\K^n_n$, a supporting hyperplane $H_{u,h_K(u)}$ of $K$ is called $(n-p-1)$-extreme \cite[p.\ 85]{Sch} if its outer normal vector $u \in \mathbb{R}^n \setminus \{0\}$ cannot be represented as a sum of $n-p+1$ linearly independent outer normal vectors at the same boundary point of $K$. 

\begin{defn}[{\cite[p.\ 86]{Sch}}]\label{def: tang body}
Let $K,E \in \mathcal{K}_n^n$ be such that $E \subseteq K$, and let $p \in \{0,\ldots,n-1\}$. Then $K$ is a {\em $p$-tangential body of $E$} if each $(n-p-1)$-extreme supporting hyperplane of $K$ is a supporting hyperplane of $E$.
\end{defn}

We observe that $K$ is a tangential body of $E$ in the sense of Definition \ref{d: tang body} if and only if $K$ is just an $(n-1)$-tangential body of $E$ \cite[pp.\ 86, 149]{Sch}.

\begin{rem}\label{rem: tang body}
Let $E \in \K_n^n$.
\begin{itemize}\itemsep5pt
\item[(i)]
Every $p$-tangential body of $E$ is also a $q$-tangential body of $E$ whenever $0 \le p<q \le n-1$. 
\item[(ii)] The only $0$-tangential body of $E$ is $E$ itself. 
\item[(iii)]
A tangential body $K$ of $E$ belongs to the class $\r_p$ if and only if $K$ is an $(n-p-1)$-tangential body of $E$ \cite[Theorem~1.3]{HCS}.
\end{itemize}
\end{rem}

\begin{proof}[Proof of Proposition \ref{p:i,i+1 quot}]

Let $0\leq i\leq n-2$,  and let $K\in\mathcal{K}_n^n$ lie in the class $\r_i$.
By Lemma~\ref{l:der Wi^1/n-i} and the assumption of $K\in\r_i$, the derivative $\frac{\dlat}{\dlat \lambda} \left(\W_i(\lambda)^{\frac{1}{n-i}}\right)$ exists and satisfies \eqref{e:der Wi^1/n-i}.
Theorem~\ref{gral BM th} and Lemma~\ref{fact:parallel_concave} ensure that $\lambda \mapsto \W_i(\lambda)^{\frac{1}{n-i}}$ defines a concave and differentiable function on $(-\inr(K;E),\infty)$. Thus, the derivative $\frac{\dlat}{\dlat \lambda} \left(\W_i(\lambda)^{\frac{1}{n-i}}\right)$ is monotonically decreasing (see e.g.\ \cite{Royden}).
Hence, by \eqref{e:der Wi^1/n-i}, $\I_i(\lambda)=\left(\frac{\dlat}{\dlat \lambda} \left(\W_i(\lambda)^{\frac{1}{n-i}}\right)\right)^{n-i}$ decreases, too.

We notice that the statements (i)-(iv) are the exact analogues of the statements (i)-(iv) in Theorem \ref{thm:overline{I}}, with the only addendum of $K$ being an $(n-i-1)$-tangential body of $E$ instead of just a tangential body of $E$. However, these two assertions are equivalent for bodies $K \in \r_i$ by Remark~\ref{rem:  tang body}. Hence the proof of the equivalences of (i)-(iv) can follow the respective steps of the proof of Theorem~\ref{thm:overline{I}}. 

For the last part of the proof, we can proceed in analogy to the proof of Theorem~\ref{thm:overline{I}}, too. Indeed, note that from the homogeneity of quermassintegrals, once obtained $K_\lambda=(1+\lambda)E$, it follows that $\I_i(\lambda)=\vol(E)$ by using that $\W_j(E;E)=\vol(E)$ for any $0\leq j\leq n$.
\end{proof} 

We observe that for $i=0$, we recover Theorem \ref{thm:overline{I}} (notice the multiplicative constant $n$ from \eqref{e: rel Sa}), since the class $\r_0$ consists of all convex bodies. 

Now we obtain the announced improvement of \cite[Proposition 4.3]{HCSq} (i.e., of Proposition~\ref{prop:  HCSq}) as a corollary.

\begin{prop}\label{prop:  newHCSq}
Let $K,E \in \mathcal{K}_n^n$, let $0 \le i < j < n$, and suppose that $K$ belongs to the class $\r_{j-1}$. Then the function
\[
\I_{i,j}(\lambda)= \frac{\W_j(K_\lambda;E)^{n-i}}{\W_i(K_\lambda;E)^{n-j}}
\]
is monotonically decreasing on $(-\inr(K;E),\infty)$.

Moreover, if $E$ is smooth, the following are equivalent for all $-\inr(K;E) < \lambda_0 < \lambda_1 <\infty$:
\begin{itemize}\itemsep5pt
\item[(i)]
$\I_{i,j}(\lambda_0)=\I_{i,j}(\lambda_1)$,
\item[(ii)]
$K_{\lambda_0}$ is homothetic to $K_{\lambda_1}$,
\item[(iii)]
$K_{\lambda_1}$ is homothetic to an $(n-j)$-tangential body of $E$,
\item[(iv)]
$\I_{i,j}(\lambda)$ is constant on $(-\inr(K;E),\lambda_1]$.
\end{itemize}

If $E$ is smooth and $\lambda_1 > 0$, conditions (i)-(iv) are satisfied if and only if $K$ is homothetic to $E$ and, consequently, if and only if $\I_{i,j}(\lambda)=\vol(E)^{j-i}$ for all $\lambda \in (-\inr(K;E),\infty)$.
\end{prop}

\begin{proof}
Let  $K,E\in \mathcal{K}_n^n$, let $0 \le i < j < n$, and let $K$ belong to the class $\r_{j-1}$.
Since $K \in \r_{j-1}$, we know from \eqref{e: content classes ri}, that $K \in \r_{k}$ for $k=i,\ldots,j-1$. Proposition~\ref{p:i,i+1 quot} shows that all the functions $\I_k(\lambda)$, $k=i,\ldots,j-1$, decrease. Then
\[
\prod_{k=i}^{j-1} \I_k(\lambda)^{\frac{(n-i)(n-j)}{(n-k-1)(n-k)}}= \frac{\W_j(K_\lambda;E)^{n-i}}{\W_i(K_\lambda;E)^{n-j}}=\I_{i,j}(\lambda)
\]
decreases as well.

Moreover, we obtain $\I_{i,j}(\lambda_0)=\I_{i,j}(\lambda_1)$ for some $-\inr(K;E) < \lambda_0 < \lambda_1$ if and only if $\I_k(\lambda_0)=\I_k(\lambda_1)$ for $k=i,\ldots,j-1$. The characterizations of the last assertions, given in Proposition~\ref{p:i,i+1 quot}, yield the remainder of Proposition~\ref{prop:  newHCSq}.
\end{proof}

The following result is now obtained in the same way as Corollary~\ref{cor:  isop optimal} was proven using Theorem \ref{thm:  Wulff}.

\begin{cor}\label{cor:  new isop opt}
Let $0 \le i < j < n$ and let $E \in \K_n^n$ be smooth. Then a convex body $\tilde{K}\in \r_{j-1} \cap \K_n^n$ is a minimizer of the quotient $\frac{\W_j(K;E)^{n-i}}{\W_i(K;E)^{n-j}}$ among all convex bodies $K\in \r_{j-1} \cap \K_n^n$ if and only if $\tilde{K}$ is homothetic to $E$.
\end{cor}

The assumption $K \in \r_{j-1}$ in Proposition~\ref{prop:  newHCSq} and Corollary~\ref{cor:  new isop opt}, respectively, is essential in our proof. However, so far we do not have
an example of a body $K \in \K_n^n \setminus \r_{j-1}$ that does not satisfy the claims of Proposition~\ref{prop:  newHCSq} or Corollary~\ref{cor:  new isop opt}.

\begin{rem}
Theorem~\ref{thm:overline{I}} on the family $(K_\lambda)_{\lambda > -\inr(K;E)}$ of parallel bodies gave rise to the analogous Theorem~\ref{thm:  Wulff} on the families $(K(\Omega,\lambda))_{\lambda > -\inr(K;E)}$, since the last families could be interpreted as inner parallel bodies by Lemma~\ref{lem:  Wulff}. In a similar way, results of the present section imply analogues concerning the families  $(K(\Omega,\lambda))_{\lambda > -\inr(K;E)}$. Then conditions $K \in \r_i$ have to be replaced by $K(\Omega,\Lambda) \in \r_i$ for all $\Lambda>-\inr(K;E)$.
\end{rem}




\begin{thebibliography}{99}

\bibitem{Alexandrov1938} A. D. Alexandrov, Zur Theorie der gemischten Volumina von konvexen K\"orpern, III: Die Erweiterung zweier Lehrs\"atze Minkowskis \"uber die konvexen Polyeder auf beliebige konvexe Fl\"achen (in Russian), {\it Mat.\ Sbornik N. S.} {\bf 3} (1938), 27--46.
 
\bibitem{Bol43} G. Bol, Beweis einer Vermutung von H. Minkowski, {\it
Abh.\ Math.\ Sem.\ Univ.\ Hamburg} {\bf 15} (1943), 37--56.

\bibitem{BF} T. Bonnesen and W. Fenchel, {\it Theorie der konvexen
K\"orper}. Springer, Berlin, 1934, 1974. English translation: {\it Theory
of convex bodies}. Edited by L. Boron, C. Christenson and B. Smith. BCS
Associates, Moscow, ID, 1987.

\bibitem{Di} A. Dinghas, Bemerkung zu einer Versch\"arfung der
isoperimetrischen Ungleichung durch H. Hadwiger, {\it Math.\ Nachr.} {\bf
1} (1948), 284--286.

\bibitem{Di49} A. Dinghas, \"Uber eine neue isoperimetrische Ungleichung
f\"ur konvexe Polyeder, {\it Math.\ Ann.} {\bf 120} (1949),
533--538.


\bibitem{Domokos} G. Domokos, Z. L\'angi, The isoperimetric quotient of a convex body decreases monotonically under the Eikonal abrasion model, {\it Mathematika} {\bf 65} (1) (2019), 119--129.


\bibitem{GK} R. Gardner, M. Kiderlen, A solution to Hammer's X-ray
reconstruction problem, {\it Adv.\ Math.} {\bf 214} (1) (2007), 323--343.


\bibitem{Grub} P. M. Gruber, {\it Convex and Discrete Geometry}.
Springer, Berlin Heidelberg, 2007.

\bibitem{Ha55} H. Hadwiger, {\it Altes und Neues \"uber konvexe
K\"orper}. Birkh\"auser Verlag, Basel und Stuttgart, 1955.


\bibitem{Ha57} H. Hadwiger, {\it Vorlesungen \"uber Inhalt,
Oberfl\"ache und Isoperimetrie}. Springer-Verlag,
Berlin-G\"ottingen-Heidelberg, 1957.

\bibitem{HCS} M. A. Hern\'andez Cifre, E. Saor\'\i n, On differentiability
of quermassintegrals, {\it Forum Math.} {\bf 22} (1) (2010), 115--126.

\bibitem{HCSq} M. A. Hern\'andez Cifre, E. Saor\'\i n, Isoperimetric relations for inner parallel bodies, 
\url{https://arxiv.org/abs/1910.05367}  (submitted).


\bibitem{HW} M. Henk, J. M. Wills, A Blichfeldt-type inequality for the
surface area, {\it Monatsh.\ Math.} {\bf 154} (2008), 135--144.

\bibitem{Larson} S. Larson, A bound for the perimeter of inner parallel bodies, {\it J.\ Funct.\ Anal.} {\bf 271} (2016),
610--619.

\bibitem{Lindeloef1869} L. Lindel\"of, Propri\'et\'es g\'en\'erales des poly\`edres qui, sous une \'etendue superficielle donn\'ee, renferment le plus grand volume, {\it Bull.\ Acad.\ Sci.\ St.\ P\'etersbourg} {\bf 4} (1869), 257--269. Extract: {\it Math.\ Ann.} {\bf 2} (1870), 150--159.


\bibitem{Mat} G. Matheron, La formule de Steiner pour les \'erosions, {\it
J.\ Appl.\ Prob.} {\bf 15} (1978), 126--135.

\bibitem{Minkowski1911} H. Minkowski, Allgemeine Lehrs\"atze \"uber die convexen Polyeder, {\it Nachr.\ Ges.\ Wiss.\ G\"ottingen}, 198--219. {\it Gesammelte Abhandlungen}, vol.\ II, pp.\ 103--121, Teubner, Leipzig, 1911.

\bibitem{RV} J. H. Rolfes, F. Vallentin, Covering compact metric spaces
greedily, {\it Acta Math.\ Hung.} {\bf 155} (1) (2018), 130--140.

\bibitem{Royden} H. L. Royden, {\it Real Analysis}, Third Edition, Macmillan Publishing Company, New York, 1988.

\bibitem{SYphd} J. R. Sangwine-Yager, {\it Inner Parallel Bodies and
Geometric Inequalities}. Ph.D.\ Thesis Dissertation, University of
California Davis, 1978.

\bibitem{Sch} R. Schneider, {\it Convex Bodies: The Brunn-Minkowski
Theory}, Second Expanded Edition, Cambridge University Press, Cambridge, 2014.


\bibitem{Ste40} J. Steiner, \"Uber parallele Fl\"achen, {\it Monatsber.\ Preu{\ss}.\ Akad.\ Wiss.}, Berlin (1840),
114--118. Also: {\it Gesammelte Werke}, vol.\ 2, pp.\ 171--176, Reimer, Berlin, 1882.

\end{thebibliography}
\end{document}